\documentclass{amsart}


\usepackage{amsmath,amssymb, amstext,amsthm,amsfonts}
\usepackage[dvips]{graphicx}
\usepackage{enumitem}

\thanks{Partially supported by CNPq, FAPERJ and PRONEX/DS from Brazil.}

\newtheorem{theorem}{Theorem}[section] 
\newtheorem{lemma}[theorem]{Lemma}     
\newtheorem{corollary}[theorem]{Corollary}
\newtheorem{Prop}[theorem]{Proposition}
\newtheorem{defi}[theorem]{Definition}

\newtheorem{example}[theorem]{Example}
\newtheorem{thm}{Theorem}

\newcommand{\T}{\mathbb{T}}

\newcommand{\de} {\delta}       \newcommand{\De}{\Delta}

\newcommand{\vep}{\varepsilon}

\newcommand{\la} {\lambda}

\newcommand{\si} {\sigma}

\newcommand{\om} {\omega}       \newcommand{\Om}{\Omega}

\newcommand{\Z}{\mathbb{Z}}
\newcommand{\N}{\mathbb{N}}
\newcommand{\R}{\mathbb{R}}

\newcommand{\SK}{{\mathcal K}}

\title[On the Space of Iterated Function Systems and Their Topological Stability]
 {On the Space of Iterated Function Systems and Their Topological Stability} 

\author{A. Arbieto, A. Trilles}

\address{Instituto de Matem\'atica, Universidade Federal do Rio de Janeiro, P. O. Box 68530, 21945-970 Rio
de Janeiro, Brazil.}
\email{arbieto@im.ufrj.br, trilles@matematica.ufrj.br}

\begin{document}
\maketitle

\begin{abstract}
	We study iterated function systems (IFS) with compact parameter space. We show that the space of IFS with phase space $X$ is the hyperspace of the space of self continuous maps of $X$. With this result we obtain that the Hausdorff distance is a natural metric for this space which we use to define topological stability.

Then we prove, in the context of IFS, the classical results showing that shadowing property is a necessary condition for topological stability and shadowing property added to expansiveness are a sufficient condition for topological stability. To prove these statements, in fact, we use a stronger type of shadowing, called concordant shadowing property.

	We also give an example showing that concordant shadowing property is truly different than the traditional definition of shadowing property for IFS.

\end{abstract}


\section{Introduction} 
\label{intro}
In 1981, Hutchinson~\cite{Hutchinson} introduced the Iterated Function Systems (IFS) as a way of studying fractals. In that case he studied only hyperbolic IFS with finite parameter space, a finite collection of contractions. His theory and fractal theory was disseminated by different books, as~\cite{Barnsley, FALCONER}.

He realized that the study of the omega-limit set of a collection of maps is connected with the iteration of compact sets. It is like a collective dynamics and he found that the base space was already been study by topologists, the so called hyperspace of a compact metric space, using the Hausdorff metric.

Eventually, it was realized that the theory of IFS can be seen as the action of a atomic measure on the space of dynamics over the phase space, generating a random dynamical system. So, a natural question arises: Instead of an atomic measure (related with finitely maps) could be used a Radon measure with compact support? In other words, could be the parameter space  compact instead of finite?

This was pursuit by many authors as in~\cite{Lewellen,Mendivil}, and eventually Arbieto, Junqueira e Santiago~\cite{AJS} obtained several results in this setting assuming very weak sources of contractions. More recently, Melo~\cite{ITALO} generalized this in his thesis.

The theory of dynamical systems had a boost in mathematics with the advent of hyperbolic theory due to Smale \cite{Smale}. The horseshoe became the paradigmatic example and had two important topological dynamical features: expansiveness and the shadowing. These two notions were extensively studied by many mathematicians such as Das, Kato,  Sakai, Thakkar and many others. One of theirs best feature was that they were heavily used in the stability of hyperbolic differentiable dynamics, see \cite{Lan_Wen}.

It turns out, that in topological dynamics, this also leads to some type of stability (nowadays called topological stability), see~\cite{Hiraide}. Moreover, it was shown that the shadowing property is a necessary condition to topological stability. This was a seminal result that gives rise to the study of shadowing-type properties and even stronger forms of stability, like the Gromov-Hausdorff Stability by Arbieto and Morales, see~\cite{GromovHausdorff}.

Naturally, this notion was exploited for IFS by some authors, see~\cite{IRANIANOS}. However, as far as we understand it is not quite precise. Moreover, the study is done in the finite case. So, the purpose of this work is to clarify this issues and to prove in the case of a compact parameter space.

For this, we use a stronger type of shadowing for IFS and we propose another definition for topological stability. These definitions permit us to obtain our main results showing that this stronger shadowing is a necessary condition for topological stability, extending the original result for maps, which can be seen in~\cite{PILYUGIN}, to the context of IFS and we also show that this type of shadowing added to the expansiveness are sufficient condition for topological stability.

The notion of topological stability deals with proximity of objects, in our case IFS. Looking for a way to measure distance between two IFS we notice that the space of IFS with a fixed phase space $X$ is the hyperspace of the space of self continuous maps of $X$. With this result we obtain that the Hausdorff metric is a natural complete metric for the space of IFS. 
\section{Definitions}

Consider $(X,d)$ a compact metric space and the space $C^0(X)$ of the self continuous maps of $X$ with the $C^0$-metric given by:
$$d_{C^0}(f,g) = \max\{d(f(x),g(x) : \forall x \in X\}$$

A family $\{\omega_{\lambda} : \lambda \in \Lambda\} \subset C^0(X)$ such that $\Lambda$ is a compact metric space and $\omega: \Lambda \times X \rightarrow M $, given by $\omega(\lambda,x):=\omega_{\lambda}(x)$ is a continuous map, is said to be an \textit{Iterated Function System} (IFS for short), and we call $\omega$ its general map.
The space $\Lambda$ is called the \textit{parameter space} and $X$ is called the \textit{phase space} of this IFS. 
We will often refer to an IFS by its general map but it is important observe that different general maps can represent the same IFS.

The space $\Lambda^\N$ endowed with the product topology will be denoted by $\Omega$. For each $\sigma = (\lambda_1,\lambda_2, ...) \in \Omega$ and $k \in \N$ we will denote the map $\omega_{\sigma_k} := \omega_{\lambda_k} \circ ... \circ \omega_{\lambda_1}$ and $\omega_{\sigma_0} := id$. A sequence $\{x_n\} \subset X$ is called a \textit{chain} (it can also be found as \textit{branch of orbit} in the literature) for the IFS $\omega$ if for each $n$ there exists $\lambda_n \in \Lambda$ such that $x_n =\omega_{\lambda_n}(x_{n-1})$. 

Since any finite set with the discrete metric is a compact metric space, we observe that any IFS with finitely many partial maps is automatically included in our definition. 

For $(X,d)$ a metric space not necessarily compact, we will denote $\SK(X)$ as the collection of all  nonempty compact subsets of $X$ and call it the \textit{Hyperspace of X}. We endow it with the Hausdorff
metric as follows. Let $d(x,F) = \inf\{d(x, y): y \in F\}$. The Hausdorff metric is given by:
$$d_H(A,B)= \max \left\{\sup_{a \in A}d(a,B), \sup_{b \in B} d(b,A)\right\}$$ 

If $(X,d)$ is a complete (resp. compact) metric space, it can be proved (see~\cite{Barnsley}) that $(\SK(X), d_H)$ is also a complete (resp. compact) metric space. With these facts it was possible to show that every hyperbolic IFS has a invariant attractor, in~\cite{AJS} the result was extended to weakly hyperbolic IFS and then after in~\cite{ITALO} extended the result for $\mathbb{P}$-weakly hyperbolic IFS.

In our study we noticed that the IFS and the Hyperspace are even more related. Our first result in this work is to explicit the relation between them which is exposed in the next theorem.

\begin{thm}
	The space of IFS with phase space $X$ is the Hyperspace of $C^0(X)$.
\end{thm}

This theorem permit us to conclude that the space of IFS with phase space $X$ is a complete metric space with the Hausdorff metric, and it comes from the fact that since $X$ is a compact metric space, then $C^0(X)$ is a complete space with $C^0$-topology. To fix the notation, from now on, given $\omega: \Lambda \times X \to X$ and $\tilde{\omega}:\tilde{\Lambda} \times X \to X$ both IFS, we will denote $d_H(\omega,\tilde{\omega}):= d_H(\{\omega_\lambda : \lambda \in \Lambda\}, \{\tilde{\omega}_\lambda : \lambda \in \tilde{\Lambda}\})$. 

This topology on the space of IFS leads us to questions about density, openness or even genericity of dynamical properties for IFS. 
As an example we can see that the set of transitive IFS is a $G_\de$ set. 

First of all, we say that an IFS is \textit{transitive} if for any $U,V$ open sets there is $\si \in \Om$ and $n$ such that $\om_{\si_n}(U) \cap V \neq \emptyset$. Let $\om$ be a transitive IFS. If we fix $U$ and $V$, for any $\tilde{\om}$ in a sufficiently small neighborhood of $\om$, there is $\tilde{\si} \in \tilde{\Om}$ such that $\tilde{\om}_{\tilde{\si}_n}(U) \cap V \neq \emptyset$. So using a countable basis of open sets of $X$ we can conclude that the set of transitive IFS is a $G_\de$ set.

\section{Shadowing Property}
For maps the shadowing property consists in guaranteeing the existence of an orbit close to a pseudo-orbit, which is a sequence similar to an orbit, but where small errors is permitted for each iterate. This notion can be translated to the context of IFS trading orbit by chain.

\begin{defi}
	Given a sequence $\{x_k\}$ in $X$ and $\delta > 0$, this sequence is said to be a \textit{$\delta$-chain}  if for each $k$ there exists $\lambda_k \in \Lambda$ such that $d(x_{k},\omega_{\lambda_k}(x_{k-1})) \leq \delta$. If the sequence is finite, we call it \textit{finite $\delta$-chain}.
\end{defi}

So, shadowing property in the context of IFS means that any $\vep>0$ there exists $\de>0$ such that any  $\de$-chain is $\vep$-close to a chain.

\begin{defi}
We say that an IFS $\omega$ has the \textit{shadowing property} if for any given $\varepsilon > 0$, there exists $\delta > 0$ such that for any $\delta$-chain $\{x_k\}$ there exists a chain $\{y_k\}$ such that $d(x_k,y_k) < \varepsilon$ for all $k\geq0$. In this case, we say that $\{y_k\}$ \textit{$(\varepsilon)$-shadows} $\{x_k\}$.
\end{defi}

\textit{Remark.} We remark that this definitions of shadowing property does not guarantee any relation between the sequence of parameters of the shadow and the sequence of parameters of the $\delta$-chain.
\vspace{0.3cm}

In fact, we give an example which has shadowing property but for some $\de$-chains it is impossible to shadow with the same sequence of parameters. Moreover, our example is constituted by rotations, so it is possible to have shadowing property in the IFS even with all partial maps not having shadowing property.

For some technical reasons, sometimes we would like to guarantee the existence of a shadow with the same sequence of parameters of the $\de$-chain. For this we use a stronger definition of shadowing for IFS presented in~\cite{Glavan2}.

\begin{defi}
We say that an IFS has \textit{concordant shadowing property} if for any $\varepsilon>0$ there exists $\delta>0$ such that every $\delta$-chain can be $\varepsilon$-shadowable by a chain with the same sequence of parameters.
\end{defi}

In 2006, Glavan and Gutu~\cite{Glavan} showed that this stronger definition is not so restrictive despite not having used this term. They also worked in a more general setting, considering an IFS any collection of continuous maps.

They start by the large and well known class of hyperbolic IFS (which can also be found as uniformly contracting IFS).

\begin{defi}
	We say that the an IFS $\{\omega_{\lambda} : \lambda \in \Lambda\}$ is \textit{hyperbolic} if
	$$
	\sup_{\lambda \in \Lambda} \sup_{x \neq y} \frac{d(\omega_\lambda(x),\omega_\lambda(y))}{d(x,y)}<1
	$$
\end{defi}

\begin{theorem}[Glavan, Gutu \cite{Glavan}]\label{unif_contracting_shadowing}
Every hyperbolic IFS has concordant shadowing property.
\end{theorem} 

The second class studied by them goes in the opposite direction of the first one. In this case any pair of points move away uniformly from each other by one iteration of any partial map.

\begin{defi}
We say that an IFS $\{\omega_{\lambda} : \lambda \in \Lambda\}$ is \textit{uniformly expanding} if
$$
 \inf_{\lambda \in \Lambda } \inf_{x \neq y} \frac{d(\omega_\lambda(x),\omega_\lambda(y)}{d(x,y)}>1
$$
\end{defi}

\begin{theorem}[Glavan, Gutu \cite{Glavan}]\label{expanding_shadowing}
	Every IFS uniformly expanding with all partial maps being surjective has concordant shadowing property.
\end{theorem}

\section{Topological Stability}
Topological Stability means essentially that the behavior of an IFS remains for another IFS sufficiently close to the initial one, in a topological viewpoint. 

Differently than an orbit of a point for a map, for an IFS a point can have infinitely many chains, which together consist the orbit of the point. It is impossible compare any two chains between any different IFS, our goal is to analyze similar chains in similar IFS, and to be more precise we introduce the notion of \textit{$\delta$-compatibility}

\begin{defi}
	For $\omega$, $\tilde{\omega}$ two IFS and $\sigma$, $\tilde{\sigma}$ sequences with the same length in each parameter space, we say that the pair $(\sigma,\tilde{\sigma})$ is \textit{$\delta$-compatible} if for all $k$ we have:
	\begin{gather*}
	d_{C^0}(\omega_{\lambda_k},\tilde{\omega}_{\tilde{\lambda}_k})<\delta
	\end{gather*}
\end{defi}

\begin{defi}
We say that an IFS $\omega$ is \textit{topologically stable} if given $\varepsilon > 0 $, there exists $\delta > 0$ such that if $\tilde{\omega}$ is an IFS and $d_H(\omega,\tilde{\omega}) \leq \delta$, then for each $(\sigma,\tilde{\sigma})$ $\delta$-compatible there exists a continuous map $h: X \to X$ with the following properties:
\begin{enumerate}[label=(\roman*)]
	
	\item $d_{C^0}(\omega_{\sigma_k} \circ h,\tilde{\omega}_{\tilde{\sigma}_k})<\varepsilon$ for all $k \in \mathbb{N}$
	
	\item $d_{C^0}(h,id) < \varepsilon$
	
\end{enumerate}
\end{defi}

One of our main results, extending and clarifying the work of Rezaei and Nia in~\cite{IRANIANOS} is to show that a consequence of topological stability is the shadowing property, actually we go further and we prove that topological stability implies concordant shadowing property for IFS having a smooth compact manifold as phase space.

\begin{thm}
	Every topologically stable IFS with a smooth compact manifold as phase space has concordant shadowing property.
\end{thm}

For maps a converse for this theorem can be obtained by adding the hypothesis of expansiveness, we expected that this should be true for IFS too. 

Expansiveness, for maps, means essentially that for any two different points, their orbits move away from each other at least a constant. This notion can be translated for IFS trading orbits for chains for any sequence of parameters. 

\begin{defi}
	We say that an IFS $\omega$ is $expansive$ if there exists a constant $\eta>0$, called expansivity constant, such that for any $\sigma \in \Omega$ if $x,y \in X$ satisfy $d(\omega_{\sigma_n}(x),\omega_{\sigma_n}(y)) \leq \eta$ for all $n \in \mathbb{N}$, then $x=y$. We usually say that $\omega$ is $\eta-expansive$.
\end{defi} 

Similarly to what have been done for maps, we can prove that expansiveness added to concordant shadowing property implies in the existence of a unique shadow for a $\de$-chain where both of them have the same sequence of parameters. Using this fact we construct for each pair of sequences $\delta$-compatible a continuous map with the properties desired for topological stability, proving the last result of this work.

\begin{thm}\label{Exp+SP}
	Every expansive IFS with concordant shadowing property is topologically stable.
\end{thm}

\textit{Remark.} For the converse of the theorem the phase space is not required to be a manifold, it works for compact metric spaces in general.
\vspace{.3cm}

After give this proof, we observed that it is essentially the same of the proof given by Thakkar and Das in~\cite{DAS} the observation is that $\delta$-compatibility of sequences is equivalent to time varying maps $\delta$-close for them. 

\section{IFS of Homeomorphisms}
Although during the text we are considering IFS for maps not necessarily invertible we can also consider $Homeo(X)$ the space of self homeomorphisms of $X$ with the metric $d_h(f,g)=\max\left\{d_{C^0}(f,g), d_{C^0}(f^{-1},g^{-1})\right\}$ and replace $C^0(X)$ by $Homeo(X)$ in the theory.

In this case we need to consider $\Omega=\Lambda^{\Z^*}$, $\omega_{\sigma_k}(x)= \omega_{\lambda_k}^{-1} \circ ... \circ \omega_{\lambda_{-1}}^{-1}$ for negative $k$ and bilateral sequences instead of unilateral sequences in all definitions and in all our results.

When we gave the definition of expansiveness, actually we presented the notion of positively expansive IFS. For IFS of homeomorphisms the definition is the following.

\begin{defi}
	We say that an IFS $\omega$ is $expansive$ if there exists a constant $\eta>0$, called expansivity constant, such that for any $\sigma \in \Omega$ if $x,y \in X$ satisfy $d(\omega_{\sigma_n}(x),\omega_{\sigma_n}(y)) \leq \eta$ for all $n \in \Z$ then $x=y$. 
\end{defi}
\textit{Remark.} For shadowing property and topological stability if an IFS of homemorphisms satisfy the definitions for bilateral sequences, then it satisfies for unilateral sequences. For expansiveness it fails, since there are examples of IFS of homeomorphisms which are expansive but they are not positively expansive.

\section{Examples}
During our study we asked ourselves if it was possible to guarantee the existence of a shadow with the same sequence of parameters of the $\de$-chain in the traditional definition of shadowing property. After some time we found the following example answering negatively the question.

\begin{example}
	Consider $\T^1=\R \slash \Z$ the unit circle and the IFS given by the following general map:
	\begin{align*}
	\omega : [0,1] \times \T^1 &\to \T^1\\
	(\lambda,x) &\mapsto x + \lambda \mod 1
	\end{align*}
\end{example}

An interesting property of this IFS is that for any fixed $x \in \T^1$ we have $\omega([0,1]\times \{x\})= \T^1$ which implies that any sequence $\{x_n\}$ in $\T^1$ is a chain for the IFS. So any $\de$-chain can be shadowed by itself by changing the sequence of parameters proving that it has shadowing property. On the other hand, each partial map is a rotation and does not have shadowing property, thus if we fix a constant sequence of parameters given $\vep>0$, for any $\de>0$ there are $\de$-pseudo-orbits, that are $\de$-chains with constant sequence of parameters and cannot be shadowable by a chain with sequence of parameters.
\vspace{.3cm}

The following is an example of an expansive IFS of homeomorphisms with concordant shadowing property.
\begin{example}
	Consider $\T^2 = \R^2 \slash \Z^2$, a small $\varepsilon>0$ and $\Lambda$  the closed ball in $\R^2$ centered in zero with radius $\varepsilon$. For $x=(x_1,x_2) \in \T^2$ consider the linear map  $f(x_1,x_2)=(2x_1+x_2,x_1+x_2)$. The IFS will be given by the following general map:
	\begin{align*}
	\omega : \Lambda \times \T^2 &\to \T^2 \\
	(\lambda,x) &\mapsto f(x) + \lambda \mod 1
	\end{align*}
\end{example}

	Firstly, we observe that $\omega_0$ is the linear automorphism on the torus, which is expansive.
	
	For a fixed $\sigma=(..., \lambda_{-2}, \lambda_{-1},\lambda_1,\lambda_2,...) \in \Omega$ the structure of this IFS permits us to simplify the computation of the distance between the iterates of chains with this sequence of parameters. Let $x \in X$, the first iterate relative to $\sigma$ is:	
	\begin{gather*}
	\omega_{\lambda_1}(x)= f(x) + \lambda_1 \mod 1
	\end{gather*}
	
	The second iterate for this sequence is:
	\begin{align*}
	\omega_{\sigma_2}(x) &= \omega_{\lambda_2}(\omega_{\lambda_1}(x))\\&= f(f(x) + \lambda_1) + \lambda_2 \mod 1  \\&= f^2(x) + f(\lambda_1) + \lambda_2 \mod 1 \\&=
	\omega^2_0(x) + f(\lambda_1) + \lambda_2 \mod 1
	\end{align*} 
	
	By induction, we obtain for each positive $k$:
	\begin{gather}\label{omega1}
	\omega_{\sigma_k}(x)= \omega^k_0(x) + \sum_{i=1}^{k}f^{k-i}(\lambda_i) \mod 1
	\end{gather}
	
	Analogously, we find that:
	\begin{gather}\label{omega2}
		\omega_{\sigma_{-k}}(x)= \omega^{-k}_0(x) - \sum_{i=1}^{k}f^{-k+i}(\lambda_{-i}) \mod 1
	\end{gather}
	
	If we take $y \in \T^2$, then for any $k \in \Z$ we have the following:
	\begin{gather*}
	d(\omega_{\sigma_k}(x), \omega_{\sigma_k}(y)) = d(\omega_0^k(x),\omega_0^k(y))
	\end{gather*}
	
	Since $\omega_0$ is expansive, so is the IFS.
	
	We can also prove that the IFS has concordant shadowing property, it comes from the fact that $\omega_0$ has shadowing property. Let $\varepsilon>0$, then we have a $\delta>0$ from the shadowing property of the map $\omega_0$. To prove that the IFS has concordant shadowing property we take $\{y_k\}$ a $\delta$-chain for a sequence $\sigma=(...,\lambda_{-2},\lambda_{-1}, \lambda_1,\lambda_2,...) \in \Omega$, then we can construct a sequence $\{x_k\}$ such that for any $k>0$ we have:
	\begin{gather}\label{shadowing1}
	y_k= x_k - \sum_{i=1}^{k} f^{k-i}(\lambda_i) \mod 1
	\end{gather}

	\begin{gather}\label{shadowing1n}
	y_{-k}= x_{-k} + \sum_{i=1}^{k} f^{-k+i}(\lambda_{-i}) \mod 1
	\end{gather}
	
	For $k>0$, if we look for $\omega_{\lambda_{k+1}}(y_k)$, from the linearity of $f$ we have:
	\begin{align}\label{shadowing2}
	\omega_{\lambda_{k+1}}(y_k) &= f(y_k) - \lambda_{k+1} \mod 1 \nonumber \\
	&= f(x_k) - \left(\sum_{i=1}^{k} f^{k+1-i}(\lambda_i)\right) - \lambda_{k+1} \mod 1\\
	&=\omega_0(x_k)- \sum_{i=1}^{k+1} f^{k+1-i}(\lambda_i) \mod 1\nonumber
	\end{align}
	
	Analogously we have:
	\begin{gather}\label{shadowing2n}
	\om_{\lambda_{-k+1}}(y_{-k})= \om_0(x_{-k}) + \sum_{i=1}^{k-1} f^{-k+i-1}(\lambda_{-i}) \mod 1
	\end{gather}
	
	On the other hand, using (\ref{shadowing1}) and (\ref{shadowing1n}) we can explicit the expression of $y_{k+1}$ and $y_{-k+1}$.	
	So, from (\ref{shadowing2}) and (\ref{shadowing2n})  we obtain that for any $k \in \Z$:
	\begin{gather*}
	d(x_{k+1}, \omega_0(x_k))=d(y_{k+1}, \omega_{\lambda_{k+1}}(y_k))<\delta
	\end{gather*}
	
	Thus, $\{x_k\}$ is a $\delta$-pseudo orbit of $\om_0$ and by the shadowing property there is $z_0$ such that for any $k \in \Z$ we have:
	\begin{gather*}
	d(x_k,\omega_0^k(z_0))<\varepsilon
	\end{gather*}
	
	From the expression obtained for $\omega_{\sigma_k}(x)$ in (\ref{omega1}) and (\ref{omega2}), we can conclude that:
	\begin{gather*}
	d(y_k,\omega_{\sigma_k}(z_0))=d(x_k,\omega_0^k(z_0))<\varepsilon
	\end{gather*}
	Then it is proved the concordant shadowing property.

\textit{Remark.} Instead of the IFS be expansive, since $\om_0$ is not positively expansive, the IFS also cannot be. 

\section{Proof of Theorem 1}
A first observation is that the family of partial maps is uniformly equicontinuous, and it comes directly from the fact that the general maps is continuous.

\begin{Prop}
	The family of partial maps of an IFS is uniformly equicontinuous.
\end{Prop}

\begin{proof}[Proof]
	Let $\lambda \in \Lambda$. Since $\omega$ is continuous and $\Lambda \times X$ is compact, then $\omega$ is uniformly continuous which means that for any given $\varepsilon>0$ there exists $\delta>0$ such that for any $d(x,y)=d\left((\lambda,x),(\lambda,y \right)<\delta$ implies that ${d\left( \omega(\lambda,x),\omega(\lambda,y) \right)=d\left( \omega_\lambda(x),\omega_\lambda(y) \right)<\varepsilon}$.
\end{proof}

With this fact, we can construct an auxiliary continuous function and using it, we show that every IFS is a compact subset of $C^0(X)$ 

\begin{proof}[Proof of Theorem 1]
	To see that $K \in \SK(C^0(X))$ is an IFS we just need to consider the parameter space as $K$ with the $C^0$-metric. The other continence requires a little bit more.
	
	Let $\{\omega_\lambda : \lambda \in \Lambda\}$ be an IFS, we want to show that it is a compact subset of $C^0(X)$. For this, we define $\varphi : \Lambda \to C^0(X)$ given by $\varphi(\lambda)=\omega_\lambda$ and we claim that $\varphi$ is continuous.
	
	Let $\{\lambda_n\}$ be a sequence in $\Lambda$ converging to $\lambda$. By the continuity of the first variable of $\omega : \Lambda \times X \to X$, we obtain that $\{\omega_{\lambda_n}\}$ converges pointwise to $\omega_\lambda$. As $\{\omega_\lambda : \lambda \in \Lambda\}$ is uniformly equicontinuous, so is $\{\omega_{\lambda_n}\}$. Since $X$ is compact, we obtain that the convergence is uniform, which implies convergence in the $C^0$-topology and consequently the continuity of $\varphi$. Therefore, $\varphi(\Lambda)=\{\omega_\lambda : \lambda \in \Lambda\}$ is compact and that completes the proof.
\end{proof}

\section{Proof of Theorem 2}
For some technical reasons sometimes we will need the phase space to be a smooth compact manifold, in those cases we will replace $X$ by $M$ and $d$ will be a Riemannian metric on $M$, and every time we refer to a manifod we will be in this context.

Similarly to the definition of shadowing property, we can give another definition which in advance seems to be weaker, permitting to obtain shadows only for finite $\de$-chains.

\begin{defi}
	We say that an IFS $\omega$ has the \textit{finite shadowing property} if for any given $\varepsilon > 0$, there exists $\delta > 0$ such that for any finite $\delta$-chain $\{x_1,...,x_n\}$ there exists a chain $\{y_k\} $ such that $d(x_k,y_k)<\varepsilon$, for $k=1,...,n$.
\end{defi}

As we are considering $\Lambda$ and $X$ compact, we have an equivalence between these two definitions. Clearly shadowing property implies finite shadowing property and the converse is given in the following lemma.

\begin{lemma}
	If an IFS has finite shadowing property, then it has shadowing property.
\end{lemma}
\begin{proof}[Proof]
	Let $\varepsilon>0$, by hypothesis there exists $\delta>0$ such that every finite $\delta$-chain can be $\varepsilon$-shadowable. Let $\{x_k\}$ be a $\delta$-chain, then for each natural number $i$ there exists $y_i \in X$ and $\sigma^i=(\lambda_1^i, \lambda_2^i,...) \in \Omega$ such that: 
	\begin{equation} \label{eq_1}
	d(\omega_{\sigma_j^i}(y_i),x_j)<\varepsilon, j=0,...,i.
	\end{equation} 
	
	By compactness of $\Lambda$ and $X$, we can assume that $\{y_i\}$ converges to $y$ and $\{\lambda_k^i\}$ converges to $\lambda_k$ for all $k$. From this we can construct $\sigma = (\lambda_1, \lambda_2, ...)$.
	Fixed $n$, we have that $(\lambda_1^i,...,\lambda_n^i)$ converges to $(\lambda_1,...,\lambda_n)$. As $\omega$ is continuous, $\omega_{\sigma_n^i}$ converges to $\omega_{\sigma_n}$ and by (\ref{eq_1}) $d(\omega_{\sigma_n}(y),x_n)<\varepsilon$.
	Thus, as $n$ is arbitrary, we have that $\{\omega_{\sigma_k}(y)\}$ $\varepsilon$-shadows $\{x_k\}$ and the IFS has shadowing property. 
\end{proof} 

If we assume the dimension of the manifold to be at least $2$, starting with the identity map and making a local translation using a bump function we can commute points sufficiently close obtaining a diffeomorphism close to the identity in $C^0$-topology. In the uni-dimensional case a simple rotation can do this, so we have the following lemma.
\vspace{.5cm}

\begin{lemma}\label{lemma_homogeneidade}
	For any given $\varepsilon>0$, there exists $\de>0$ such that if $x \neq y$ and $d(x,y)<\de$ , then there exists a diffeomorphism $f$ such that:
	\begin{enumerate}[label=(\roman*)]		\item $d_{C^0}(f,id)<\varepsilon$		
	\item $f(x)=y$
		
	\end{enumerate}
\end{lemma}

The following is the key lemma for the proof of Theorem 2. Having an IFS and a $\delta$-chain we shall construct another IFS as close as wanted to the one we have. We also obtain a chain for this new IFS close to the initial $\delta$-chain and such that their sequences are $\delta$-compatible, with this, the theorem becomes easy.

\begin{lemma}
	Let $\omega$ be an IFS with a manifold $M$ as phase space. Given $\Delta > 0$, there exists $\delta>0$ such that if $\{x_0,...,x_n\}$ is a finite $\delta$-chain with sequence $\sigma$, then there exists an IFS $\tilde{\omega}$ satisfying $d_H(\omega, \tilde{\omega}) < \Delta$ and a sequence $\tilde{\sigma}$ such that $(\sigma,\tilde{\sigma})$ is $\Delta$-compatible  satisfying $d(\tilde{\omega}_{\tilde{\sigma}_k}(x_0), x_k)<\Delta$ for $k=0,...,n$.
\end{lemma}
\begin{proof}[Proof]
	For $\De>0$, from lemma~\ref{lemma_homogeneidade} we obtain $\de>0$ such that if $d(x,y)<3\de$ then there exists a diffeomorphism $f$ such that:
	
	\begin{enumerate}[label=(\roman*)]
		
		\item $d_{C^0}(f,id)<\Delta$
		
		\item $f(x)=y$
		
	\end{enumerate}
	
	We assume $3\de<\De$. Let $\{x_0,...,x_n\}$ be a $\delta$-chain. Since it is finite, using triangle inequality we can take $\{x_0=y_0,...,y_n\} $ a $3\delta$-chain such that:
	\begin{enumerate}[label=(\roman*)]
		
		\item $d(x_k,y_k)<\Delta$, $k=0,..,n$
		
		\item $\omega_{\lambda_{k+1}}(y_k) \neq y_{k+1}$, $k=0,...,n-1$
	\end{enumerate}
	
	So for $k=0,...,n-1$ there exists $h_k$ such that $d_{C^0}(h_k,id)<\Delta$ and $h_k(\omega_{\lambda_{k+1}}(y_k))=y_{k+1}$.
	
	We define $\tilde{\Lambda}=\{0,...,n-1\} \times \Lambda$ and $\tilde{\omega}:\tilde{\Lambda} \times M \rightarrow M$ where $\tilde{\omega}(k,\lambda,x):=h_k \circ \omega_{\lambda}(x)$. As $\omega$ is continuous, so is $\tilde{\omega}$ and then it is a general map of an IFS. We claim that for each $k \in \{0,...,n-1\}$ and $\lambda \in \Lambda$ we have $d_{C^0}(\tilde{\omega}_{(k,\lambda)},\omega_{\lambda}) < \Delta$. 
	
	Let $x \in M$, then we have:
	\begin{equation} \label{eq_lemma6}
	d(\tilde{\omega}_{(k,\lambda)}(x),\omega_{\lambda}(x))=d(h_k(\omega_{\lambda}(x)),\omega_{\lambda}(x)) \leq d_{C^0}(h_k,id) <\Delta
	\end{equation}
	As $x$ was arbitrary, we have $d_{C^0}(\tilde{\omega}_{(k,\lambda)},\omega_{\lambda}) < \Delta$ and then $d_H(\omega,\tilde{\omega})<\Delta$.
	
	We define $\tilde{\sigma}=(\tilde{\lambda}_1,...,\tilde{\lambda}_n)$, where $\tilde{\lambda}_k := (k,\lambda_k)$. Thus we have $\tilde{\omega}_{\tilde{\sigma}_{k}}(x_0) = y_k$ and from (\ref{eq_lemma6}) $(\sigma,\tilde{\sigma})$ is $\Delta$-compatible.
	
\end{proof}

\begin{lemma}\label{thm_FSP}
	Every topologically stable IFS with a manifold as phase space has finite shadowing property. Moreover, if the sequence of parameters of the finite $\de$-chain has $n$ elements, then they coincide with the firsts $n$ elements of the sequence of parameters of the shadow.
\end{lemma}

\begin{proof}[Proof]
	Let $\omega:\Lambda \times M \to M$ be an IFS topologically stable with $dimM \geq2$. For a given $\varepsilon >0$, from the definition of topological stability we obtain a $\Delta >0$ such that if $\tilde{\omega}$ is an IFS with $d_H(\omega,\tilde{\omega})< \Delta$, for each $(\sigma,\tilde{\sigma})$ $\Delta$-compatible there exists a continuous map $h:M \rightarrow M$ with the following properties:
	\begin{enumerate}[label=(\roman*)]
		
		\item $d_{C^0}(\omega_{\sigma_k} \circ h,\tilde{\omega}_{\tilde{\sigma}_k})<\frac{\varepsilon}{2}$ for all $k \in \mathbb{N}$
		
		\item $d_{C^0}(h,id) < \frac{\varepsilon}{2}$
		
	\end{enumerate}
	
	We assume $\Delta < \frac{\varepsilon}{2}$.
	
	Let $\delta < \frac{\Delta}{6}$ and $\{x_0,...,x_n\}$ a $\delta$-chain with sequence $\sigma$, then from the previous lemma there exists $\tilde{\omega}$ with $d_H(\omega,\tilde{\omega})<\Delta$, $\tilde{\sigma}$ such that $(\sigma,\tilde{\sigma})$ is $\Delta$-compatible and $y_0$ such that $d(\tilde{\omega}_{\tilde{\sigma}_k}(y_0),x_k)<\Delta$ for $k=0,...,n$.
	
	We will consider, by now, $\sigma$ and $\tilde{\sigma}$ infinite by complete with $\lambda_k = \lambda_1$ and $\tilde{\lambda}_k=(1,\lambda_1)$ for $k \geq n$. We remark that $(\sigma, \tilde{\sigma})$ is still $\delta$-compatible.
	
	So, we obtain $h:M \rightarrow M$ a continuous map with the properties mentioned above.
	
	We consider $z_0 = h(y_0)$, then $\{\omega_{\sigma_k}(z_0)\}$ is clearly a chain for $\omega$ and we observe that for $k=1,...,n$:
	\begin{align*} 
	d(x_k,\omega_{\sigma_k}(z_0)) & =d(x_k, \omega_{\sigma_k}(h(y_0))  \\
	& \leq d(x_k, \tilde{\omega}_{\tilde{\sigma}_k}(y_0)) + d(\tilde{\omega}_{\tilde{\sigma}_k}(y_0),\omega_{\sigma_k}(h(y_0)) \\
	& \leq \Delta + \frac{\varepsilon}{2}\\
	& < \frac{\varepsilon}{2} + \frac{\varepsilon}{2} = \varepsilon
	\end{align*}
\end{proof}

\begin{corollary}
	Every topologically stable IFS having a manifold as phase space has shadowing property.
\end{corollary}

the sequence of parameters of the finite $\de$-chain has $n$ elements, then they coincide with the firsts $n$ elements of the sequence of parameters of the shadow.

Since we can shadow a finite $\de$-chain coinciding its $n$ elements of sequence of parameters with the firsts $n$ elements of the shadow and $M$ is compact we can easily prove the Theorem 2.

\begin{proof}[Proof of Theorem 2]
	Let $\varepsilon>0$. Let us consider $\delta>0$ as obtained in the proof of lemma~\ref{thm_FSP} and let $\{x_k\}$ be a $\delta$-chain with with sequence $\sigma=(\lambda_1,\lambda_2,...)$. For each $n \in \mathbb{N}$ if we consider $\{x_0,...,x_n\}$ we obtain $z_n$ and $\sigma^n=(\lambda_1^n,\lambda_2^n,...)$ such that for all $0 \leq j \leq n$ we have $\lambda_j^n=\lambda_j$  and consequently
	\begin{align}\label{eq_SSP}
	d(x_k, \omega_{\sigma_k}(z_n)) =d(x_k, \omega_{\sigma_k^n}(z_n))<\varepsilon 
	\end{align}
	
	By compactness of $M$ we can consider $\{z_n\}$ convergent and $z_0$ its limit. Fixed $k \in \mathbb{N}$, from (\ref{eq_SSP}) we have that $d(x_k, \omega_{\sigma_k}(z_0))<\varepsilon$. 
	So, $\{\omega_{\sigma_k}(z_0)\}$ is a chain that $\varepsilon$-shadows $\{x_k\}$ with the same sequence.
\end{proof}

\section{Proof of Theorem 3}
As we mentioned before, we shall use the existence of a unique chain shadowing a $\de$-chain with the same sequence of parameters. 

\begin{defi}
	We say that an IFS has \textit{uniqueness shadowing property} if it has concordant shadowing property and there exists $\varepsilon>0$ such that for its respective $\delta$ from the concordant shadowing property we have that for any $\{x_k\}$ $\delta$-chain with sequence $\sigma$ there exists an unique $y$ such that $\{\omega_{\sigma_k}(y)\}$ $\varepsilon$-shadows $\{x_k\}$. 
\end{defi}

\begin{Prop}\label{uniqueness}
	If $\omega$ is an $\eta$-expansive IFS and it has concordant shadowing property, then $\omega$ has the shadowing uniqueness property.
\end{Prop}

\begin{proof}[Proof]
	Let $\varepsilon< \frac{\eta}{2}$. As $\omega$ has concordant shadowing property, we obtain $\delta>0$ such that ever $\delta$-chain is $\varepsilon$-shadowable by a chain with the same sequence. Let $\{x_k\}$ be a $\delta$-chain with sequence $\sigma$, so there exists $y$ such that $\{\omega_{\sigma_k}(y)\}$ $\varepsilon$-shadows $\{x_k\}$.
	
	Now suppose there exists $z$ such that $\{\omega_{\sigma_k}(z)\}$ $\varepsilon$-shadows $\{x_k\}$, then we have:
	\begin{gather*}
	d(y_k,z_k) \leq d(y_k,x_k) + d(y_k,z_k)< 2\varepsilon<\eta.
	\end{gather*}
	
	Thus, as $\omega$ is $\eta$-expansive $z=y$ and $\omega$ has shadowing uniqueness property.
\end{proof}

\begin{lemma}\label{lemma_expansiveness}
	If $\omega$ is an $\eta$-expansive IFS and $\sigma=(\lambda_1,\lambda_2,...)$ is a sequence, then for any given $\mu>0$ there exists $N>0$ such that if $x,y \in M$ and $d(\omega_{\sigma_n}(x),\omega_{\sigma_n}(y)) \leq \eta$ for all $n \leq N$, then $d(x,y)< \mu$.
\end{lemma}

\begin{proof}[Proof]
	Suppose that exists $\mu$ that fails the lemma. Then for each $N \in \mathbb{N}$ there exists $x_N$ and $y_N$ such that $d(\omega_{\sigma_k}(x_N),\omega_{\sigma_k}(y_N)) \leq \eta$ for all $k \leq N$ but $d(x_N,y_N) \geq \mu$. So, we obtain $\{x_N\}_{N \in \mathbb{N}}$ and $\{y_N\}_{N \in \mathbb{N}}$ and by compactness we can assume they are convergent, respectively to $x$ and $y$. Now fixed $n \in \mathbb{N}$, by continuity of the IFS we have that $\omega_{\sigma_n}(x_N)$ converges to $\omega_{\sigma_n}(x)$ and  $\omega_{\sigma_n}(y_N)$ converges to $\omega_{\sigma_n}(y)$. As $d(\omega_{\sigma_n}(x_N),\omega_{\sigma_n}(y_N)) \leq \eta$ for all $n \leq N$, we obtain that $d(\omega_{\sigma_n}(x),\omega_{\sigma_n}(y)) \leq \eta$. On the other hand, $d(x_N,y_N) \geq \mu$ for all $N \in \mathbb{N}$ wich implies $d(x,y) \geq \mu$. This contradicts the hypothesis of $\omega$ be $\eta$-expansive.
\end{proof}

\begin{proof}[Proof of Theorem 3]
	Let $\varepsilon>0$, $\omega$ be an IFS expansive with concordant shadowing property and $\eta>0$ be the expansivity constant of $\omega$.
	From the proposition~\ref{uniqueness} we obtain that $\omega$ has shadowing uniqueness property, moreover from the proof we know that any $\varepsilon<\frac{\eta}{2}$ satisfies the shadowing uniqueness property, so let us consider $\varepsilon<\frac{\eta}{3}$, from the concordant shadowing property we have $\delta>0$ such that any $\delta$-chain is uniquely $\varepsilon$-shadowable by a chain with the same sequence.
	
	Let $\tilde{\omega}$ be an IFS with $d_H(\omega,\tilde{\omega})<\delta$. Fix $\sigma \in \Omega$ and $x \in X$. Let $\tilde{\sigma}$ be a sequence such that $(\sigma,\tilde{\sigma})$ is $\delta$-compatible.
	
	Since $(\sigma,\tilde{\sigma})$ is $\delta$-compatible, we observe that  $\{\tilde{\omega}_{\tilde{\sigma}_k}(x)\}$ is a $\delta$-chain for $\omega$ with sequence $\sigma$, then there exists a unique point $y_x$ such that the chain $\{\omega_{\sigma_k}(y_x)\}$ $\varepsilon$-shadows $\{\tilde{\omega}_{\tilde{\sigma}_k}(x)\}$.
	
	We define $h:X \rightarrow X$, by $h(x):=y_x$ and we observe that from the shadowing uniqueness property $h$ is well defined and by construction $d(x,h(x))<\varepsilon$ for all $x \in X$. Thus, if $h$ is continuous, then $d_{C^0}(h,id)<\varepsilon$. We also observe that by construction $d(\omega_{\sigma_k}(h(x)),\tilde{\omega}_{\tilde{\sigma}_k}(x))<\varepsilon$ for all $k \in \mathbb{N}$ and $x \in M$. So, if $h$ is continuous, we also have $d_{C^0}(\omega_{\sigma_k} \circ h, \tilde{\omega}_{\tilde{\sigma}_k})<\varepsilon$ for all $k \in \mathbb{N}$.
	
	We claim that $h$ is continuous. Let $\mu>0$ be given. By theorem \ref{lemma_expansiveness} there exists $N \in \mathbb{N}$ such that if $x,y \in X$ and $d(\omega_{\sigma_k}(x),\omega_{\sigma_k}(y)) \leq \eta$, for all $k \leq N$ then $d(x,y)<\mu$. For each $k \in \{0,...,N\}$, $\omega_{\sigma_k}$ and $\tilde{\omega}_{\tilde{\sigma}_k}$ are continuous, as $M$ is compact, they are uniformly continuous and then for each $k$, there exists $\beta_k>0$ and $\tilde{\beta}_k>0$ such that if $d(x,y)< \beta_k$, then $d(\omega_{\sigma_k}(x),\omega_{\sigma_k}(y))< \epsilon$, and if $d(x,y)<\tilde{\beta}_k$, then $d(\tilde{\omega}_{\tilde{\sigma}_k}(x),\tilde{\omega}_{\tilde{\sigma}_k}(y))<\varepsilon$. Take $\beta= \min\{\beta_k,\tilde{\beta}_k: k=0,...,N\}$. We observe that if $d(x,y)<\beta$ then for $k=0,...,N$ we have:
	\begin{align*}
	d(\omega_{\sigma_k}(h(x)), \omega_{\sigma_k}(h(y))) &\leq d(\omega_{\sigma_k}(h(x)), \tilde{\omega}_{\tilde{\sigma}_k}(x))+ d(\tilde{\omega}_{\tilde{\sigma}_k}(x),\tilde{\omega}_{\tilde{\sigma}_k}(y))\\ &+ d(\tilde{\omega}_{\tilde{\sigma}_k}(y),\omega_{\sigma_k}(h(y)))\\
	& < \varepsilon + \varepsilon +\varepsilon < \eta
	\end{align*}
	
	Thus, $d(x,y)<\beta$ implies $d(\omega_{\sigma_k}(h(x)), \omega_{\sigma_k}(h(y)))<\eta$ for $0,...,N$, which implies $d(h(x),h(y))<\mu$ and consequently $h$ continuous and $\omega$ is topologically stable.
\end{proof}

\textit{Remark.} We remark that for this last theorem it is not required the phase space to be a manifold, it works for any compact metric space.

\bibliographystyle{plain}
\bibliography{references}

\end{document}